\documentclass[12pt,final,reqno]{amsart}
\usepackage{ntung}

\usepackage{ntung_commands}

\DeclareMathOperator{\ER}{ER}

\usepackage{thm-restate}

\begin{document}

\title{Coloring sparse random Cayley graphs}
\author{Nathan Tung}
\address{Stanford University, CA 94305, USA}
\email{ntung@stanford.edu}
\begin{abstract}
It is shown that there exists $c > 0$ so that the Cayley graph over any finite abelian group $Z$ generated by $c \log |Z|$ random elements is properly 3-colorable with high probability (as $|Z| \to \infty$). This is asymptotically tight and improves the best-known bound due to Alon of $\frac{1}{4}\log \log |Z|$ elements. It also settles the abelian case of Alon's suggestion that a bound of $c \log |G|$ may hold for any finite solvable group $G$.
\end{abstract}
\maketitle

\section{Introduction}

 Let $Z$ be an arbitrary abelian group of size $N$. Let $S \subset Z$ be formed by choosing uniformly and independently $k$ elements of $Z$ with replacement. Then $(Z,k)$ will denote the undirected random Cayley graph on vertex set $Z$ in which distinct vertices $x,y\in Z$ are adjacent if and only if $x - y \in S\cup -S$; loops are ignored and multiplicities are suppressed. The goal of this paper is to understand the typical chromatic number $\chi(Z,k)$ of $(Z,k)$ in the sparse regime $k=\Theta(\log N)$ as $N \to \infty$. This line of inquiry was initiated by Alon in \cite{chromcay}, motivated by applications to information theory and the construction of expander graphs. There it was shown that the threshold for having bounded chromatic number lies at $k = \Theta(\log N)$ when $Z$ is the cyclic group $\ZZ_N$ or an elementary $p$-group $\ZZ_p^n$. The difficulty in such a result lies in the lower bound; that is, showing that one may take $k = \Omega(\log N)$ random generators and still have bounded chromatic number with high probability. A similar result was shown by Bourgain in the infinitary setting in \cite{bourgain}, where it was shown that a random subset of the integers of roughly logarithmic density is not intersective with high probability. Alon's techniques are specialized to cyclic and elementary $p$-groups, and the general bound of $\frac 1 4 \log \log N$ for any abelian $Z$ follows from finding the largest cyclic subgroup and a projection trick. This was not expected to be tight, however, and Alon suggested in the last section of \cite{chromcay} that if $k \le 0.01 \log N$ generators are sampled from any solvable group $G$, then it seems plausible that $\chi(G,k) \le 3$ with high probability (where $\abs{G} = N$ and $(G,k)$ is defined with multiplication instead of addition). We confirm this for abelian groups.

 \begin{theorem}\label{thm:main}
    Let $Z$ be a finite abelian group of size $N$. Then $\chi(Z,\frac{1}{64} \log N) \le 3$ with high probability, i.e., with probability $1-o(1)$ as $N \to \infty$.
\end{theorem}

The result follows from two similar results with additional assumptions on the group that together cover all finite abelian groups. Denote, for $a \in \NN$,
$$
Z[a] \coloneqq |\set{z \in Z: az = 0}|.
$$
This is the number of elements of order dividing $a$. For example, if $Z = \ZZ_N$ is cyclic then $\ZZ_N[a]$ is the greatest common divisor of $a$ and $N$. If $Z[a]$ is small for every sufficiently small $a$ (the torsion-sparse case), we use large-order characters to construct a $3$-coloring. If $Z[a]$ is large for some small $a$ (the torsion-dense case), homomorphisms to a suitable small cyclic group are used instead. These cases are outlined in more detail in \Cref{sec:strat}.

The intermediate result in the torsion-dense regime is simply the same conclusion restricted to torsion-dense groups.
\begin{restatable}[Torsion-dense coloring]{theorem}{tdense}\label{thm:highrank}
     Let $Z$ be of size $N$ and suppose there exists a positive integer $a$ such that $a \le (\log N)^4$ and $Z[a] \ge N^{1/4}$. Then $\chi(Z,\frac{1}{64}\log N) \le 3$ with high probability.
\end{restatable}

In the torsion-sparse regime we are able to capture good dependence of the number of generators on the number of colors, which can be taken to grow with $N$.
\begin{restatable}[Torsion-sparse coloring]{theorem}{td}\label{thm:lowrank}
     Let $\chi = \chi(N)$ be a sequence of integers such that $3 \le \chi \le N^{1/11}$, and $k = -(9\log(1-2/\chi))^{-1} \log N$ (approximately $\frac{\chi}{18}\log N$ for large $\chi$). Suppose $Z$ is an abelian group of size $N$ such that for every positive integer $a\le (k\log N)^2$ it holds that $Z[a] < N^{1/4}$. Then $\chi(Z,k) \le \chi$ with high probability.
\end{restatable}

Little attempt has been made to optimize constants, and essentially none are sharp. The linear dependence of $\frac{\chi}{18} \log N$ for large $\chi$ should be compared to $\Theta(\chi^2 \log N)$, as it is shown in \Cref{prop:hoffman} that for some large constant $C$ the chromatic number of $(Z,C\chi^2 \log N)$ is larger than $\chi$ with high probability.

To contextualize these results, let us discuss the closely related problem of determining the independence number of random Cayley graphs. Let $\ER(N,p)$ denote the Erd\H{o}s-Rényi random graph on $N$ vertices where each edge is included independently with probability $p$, $(Z,p)$ denote the random Cayley graph formed by including each pair $\set{z,-z}$ in $S$ independently with probability $p$, and $\alpha(H)$ denote the size of the largest independent set of a graph $H$. In a series of works \cites{green, greenmorris, campos, nenadov2025remarkindependencenumbersparse}, it has been shown that $$
\alpha(\ZZ_N,p) = (1+o(1))\alpha(\ER(N,p))
$$ 
with high probability for $p > (\log N)^{-1/3 + o(1)}$, with each improvement pushing the result to smaller $p$.
Beyond cyclic groups, it is shown by Alon and Pham in \cites{alon2025randomcayleygraphsrandom} that $\alpha(Z,p) = \widetilde{O}(p^{-3/2})$ with high probability for any $p \le 1/2$. A similar result was shown for general (not necessarily abelian) groups but with a weaker bound by Conlon, Fox, Pham, and Yepremyan in \cite{conlon2024cliquenumberrandomcayley}.

Our results work in a regime far sparser than the above, where the highly dependent edges in Cayley graphs force large, structured independent sets that are not present in the Erd\H{o}s-Rényi model. Indeed, at density $p=\Theta((\log N)/N)$, our results imply that $\alpha(Z,p) = \Omega(N)$ while it is known that $\alpha(\ER(N,p)) = o(N)$ with high probability \cite{friezeind}. Explicitly, denote $p \coloneqq k/N$ and call $Z$ $p$-torsion-sparse if $Z[a] < N^{1/4}$ for all $a \le (pN\log N)^2$. By transference from $(Z,k)$ to $(Z,p)$, \Cref{thm:lowrank} implies that if $Z$ is $p$-torsion-sparse then
$$
\alpha(Z,p) \ge \frac{N}{\chi(Z,p)} = \Omega\bigp{\frac{\log N}{p}}
$$
with high probability for $(\log N)/N \ll p \ll N^{-10/11}$. \Cref{thm:main} also implies that if $p = (\log N)/(65N)$ then $\alpha(Z,p) \ge N/3$ with high probability for any abelian $Z$.

\subsection{Outline of the proof}\label{sec:strat}

A proper $3$-coloring can be produced from a homomorphism $f: Z \to \TT \coloneqq \RR/\ZZ$ or a homomorphism $f: Z \to \ZZ_M$ for some integer $M$ as follows. First partition the torus (say $\TT$) into three intervals of nearly equal length. Then color $z \in Z$ based on the interval that $f(z)$ lies in. Any $z,z' \in Z$ sharing this color are such that $f(z-z') \in [-1/3,1/3]$. If $S$ denotes the random generating set and $f(S) \cup f(-S)$ lies in the middle third of the torus, i.e. $f(S) \cup f(-S) \subset \TT \setminus [-1/3,1/3]$, then this coloring must be proper. The existence of such an $f$ is established via a second moment argument. That is, it is shown that with sufficiently few random generators the expected number of such $f$ tends to infinity, and also that the true number of such $f$ concentrates around this expectation. The family $\cF$ of homomorphisms we search among depends on the amount of small torsion in $Z$.

If $Z[a]$ is small ($< N^{1/4}$) for all polylogarithmic $a$ (the \emph{torsion-sparse} case),
many elements of $Z$ have large order. The cyclic group $\ZZ_N$ is a fundamental example, as $\ZZ_N[a] = (a,N) \le a$. For groups of this type, after fixing an identification between $Z$ and its Pontryagin dual (equivalently, a pairing $\cdot: Z\times Z\to\TT$), we take $\cF$ to be characters of the form $z \to x \cdot z$ for $x \in Z$ with somewhat large order. To close the second moment method we show that most pairs of these characters behave roughly independently using Fourier analysis and smooth approximation of intervals in $\TT$. This allows us to find $x$ such that $x \cdot s$ lands in the middle third for all $s \in S$, so we may then color by pulling back a partition as described above.

If instead $Z[a]$ is large for some small $a$ (the \emph{torsion-dense} case), then there are many homomorphisms $g: Z \to \ZZ_a$. A good example is $\ZZ_2^n$, where $Z[2] = N$. It will turn out that a subcollection of these homomorphisms are well-behaved enough to run the second moment method. Explicitly, we choose an appropriate modulus $M\mid a$ and run the method taking $\cF$ to be all surjective maps $f:Z\to\ZZ_M$. Again with high probability one of these maps all generators into the middle third.

Heuristically, the two regimes are different restrictions of a larger class of homomorphisms. A unified framework would be to let $\cF$ be all homomorphisms $Z \to \ZZ_N$ and capture pairwise dependence of $f,g \in \cF$ by the subgroup
$$
H_{f,g}:= \set{(f(z), g(z)): z \in Z} \le \ZZ_N^2.
$$
Searching among this larger class would lead to sharper constants, but it is not so clear that our desired second moment estimates always hold. We choose to prioritize convenience over sharpness, and a key point is to avoid this bookkeeping by considering subsets of $\cF$ tailored to the torsion properties of $Z$: large-order characters and homomorphisms with a small specified image, respectively. Another simplification is to stratify pairs based on coarser information than $H_{f,g}$. In the torsion-sparse case, working over $\TT$ and smoothing allows us to simply divide pairs into ``good'' and ``bad''. In the torsion-dense case, $d(f,g)$ is a coarse statistic of $H_{f,g}$ that is easier to work with (see the proof of \Cref{thm:highrank}).

The high-level approach draws from that used in \cite{piercing}. In \cite{piercing}, a second moment method is used to find an algebraic subset of a finite field vector space disjoint from a random set of points. Here algebraic subset refers to the joint vanishing set of a collection of polynomials. In the strategy above, containing $f(S)$ in the middle third forces $S$ to be disjoint from $f^{-1}([-1/3,1/3])$. This preimage may be seen as an approximate vanishing set, to be compared with the true vanishing sets in \cite{piercing}.

\subsection{Notation and conventions}

$Z$ will denote an arbitrary finite abelian group of size $N$, and $\ZZ_N$ will denote the cyclic group. With high probability means with probability $1-o(1)$ as $N \to \infty$. $\widetilde{O}$ and $\widetilde{\Omega}$ will be used with arguments that are polynomial in $N$ to suppress factors that are polylogarithmic in $N$. All logarithms are natural.

Denote $\TT \coloneqq \RR/\ZZ$, which may also be identified with the unit interval $[0,1)$ or $[-1/2,1/2)$. Denote by $\TT_m < \TT$ the discrete subgroup isomorphic to $\ZZ_m$, i.e. $\TT_m = \{0,\frac{1}{m}, \dots, \frac{m-1}{m}\}$. For $f: \TT \to \CC$ with $|f| \le 1$, define $\widehat{f}: \ZZ \to \CC$ by
$$
\widehat{f}(a) \coloneqq \int_{\TT} f(x)e(-ax) dx,
$$
where $e(\cdot) = e^{2\pi i \cdot}$. For finite abelian $Z$ we will also fix a nondegenerate symmetric bilinear pairing $Z \times Z \to \TT$ written as $x,z \to x \cdot z$. If one desires to make an explicit choice for this pairing, $Z$ can be written as a direct sum of cyclic factors $\ZZ_a$ and the pairing can act coordinate-wise as
$$
x,z \to \frac{xz}{a} \mod 1, \quad x,z \in \ZZ_a,
$$
summing all coordinates mod 1 to attain a final value in $\TT$.

$S \subset Z$ will denote the random generating set of the undirected Cayley graph where $z,z' \in Z$ are connected if and only if $z - z' \in S \cup (-S)$. $k$ will denote the number of generators sampled to form $S$, and thus is an integer even though floor and ceiling functions may be suppressed in its definition. $\chi \ge 3$ will be an integer denoting the number of colors in a proper coloring. Denote
$$
aZ \coloneqq \set{az: z \in Z}, \quad Z[a] \coloneqq \abs{\set{z: az = 0, z \in Z}}.
$$
These are the image and kernel of the homomorphism $z \to az$, so it follows that $Z[a] |aZ| = N$. $Z[a]$ is also the number of homomorphisms $Z \to \ZZ_a$, as proved in \Cref{lem:highrankbad}. The following useful identities hold for any $a$ dividing $N$ (denoted $a \mid N$):
$$
a\ZZ_N \cong \ZZ_{N/a}, \quad \ZZ_N[a] = a.
$$

\section{Reduction to the torsion-sparse and torsion-dense cases}

First we show how \Cref{thm:main} follows directly from Theorems \ref{thm:highrank} and \ref{thm:lowrank}.

\begin{proof}[Proof of \Cref{thm:main}]
    Set $k' = (9 \log 3)^{-1} \log N$, which corresponds to $\chi = 3$ in \Cref{thm:lowrank}. Then if $Z$ is such that for all $a \le (k'\log N)^2$ it holds that $Z[a] < N^{1/4}$, by \Cref{thm:lowrank} with high probability $\chi(Z,k') \le 3$. But since $k = (\log N)/64 < k'$, one may couple such that $(Z,k) \subseteq (Z,k')$ and so also with high probability $\chi(Z,k) \le 3$. If $Z$ does not satisfy this torsion-sparse assumption, then there exists $a \le (k'\log N)^2 < (\log N)^4$ such that $Z[a] \ge N^{1/4}$, and by \Cref{thm:highrank} with high probability $\chi(Z,k) \le 3$.
\end{proof}

\section{Torsion-sparse groups}

This section proves \Cref{thm:lowrank}. Let $I \coloneqq \TT \setminus [-1/\chi,1/\chi]$. We will also identify $I$ with its indicator function. Any statement involving $I$ holds for any choice of $\chi \ge 3$. For $x,y \in Z$, denote
$$
B_{x} \coloneqq \set{z \in Z: x \cdot z \in I}, \quad B_{x,y} \coloneqq B_x \cap B_y.
$$
We aim to find $x$ such that $S \subset B_x$. Toward this end, a pair $x,y \in Z$ is called
\textit{$\delta$-good} if
$$
\abs{\frac{\abs{B_{x,y}}}{N} - \frac{\abs{B_x}\abs{B_y}}{N^2}} \le 60\delta,
$$
and \textit{$\delta$-bad} otherwise. For our uses $\delta = O((\log N)^{-2})$, and when considering bounded chromatic number we will take $\delta = \Theta ((\log N)^{-2})$.
To bound the number of bad pairs we use Fourier analysis along with a smooth approximation of $I$.

\begin{proposition}[Smooth approximation]\label{prop:approx}
    For any $\delta > 0$ smaller than an absolute constant, there exists an approximation $\tau: \TT \to [0,1]$ of $I$ such that the following holds. For any point $x \in \TT$ at a distance at least $\delta$ from the boundary of $I$
    $$
    \abs{\tau(x) - I(x)} \le \delta.
    $$
    Furthermore, for any $a \in \ZZ$
    $$
    \abs{\widehat{\tau}(a)} \le e^{-a^2 \delta^3}.
    $$
\end{proposition}
The proof is deferred to Appendix \ref{sec:smoothapprox}. The sense in which we will need $\tau$ to approximate $I$ is more directly the following. The proof uses two elementary arguments that will be reused later.

\begin{lemma}[$L^1$ error bound]\label{lem:L1}
    Let $x \in Z$ be an element of order $m$. Then
    $$
    \EE_{z} \abs{\bigp{\tau-I}(x \cdot z)} \le 5\bigp{\delta + \frac{1}{m}}.
    $$
\end{lemma}
\begin{proof}
    Let $\partial I + \delta$ denote all points at distance less than $\delta$ from the boundary of $I$. Then for any $t \notin \partial I + \delta$, by \Cref{prop:approx}
    $$
    \abs{\tau(t) - I(t)} \le \delta.
    $$
    
    \begin{claim}\label{cl:uniimage}
        The map $z \to x \cdot z$ from $Z \to \TT$ is a homomorphism whose image is the discrete subgroup $\TT_m \subset \TT$. Furthermore, if $z$ is distributed uniformly in $Z$ then $x \cdot z$ is distributed uniformly in $\TT_m$.
    \end{claim}
    \begin{proof}
    By bilinearity of $\cdot$
    $$
    m(x \cdot z) = mx \cdot z = 0 \cdot z = 0,
    $$
    so $x \cdot z \in \TT_m$. The fact that the image is all of $\TT_m$ is by nondegeneracy of $\cdot$. Indeed, suppose that this is not the case, so that the image is $\TT_d$ for some $d < m$. Then, for all $z \in Z$,
    $$
    dx \cdot z = d(x\cdot z) = 0.
    $$
    But this contradicts nondegeneracy of $\cdot$, as $dx \neq 0$. Uniformity follows from the fact that fibers of $z \to x \cdot z$ are all cosets of the kernel of $z \to x \cdot z$, all of equal size $N/m$.
    \end{proof}
    Using this claim,
    \begin{align}
         &\EE_{z \in Z} \abs{\bigp{\tau-I}(x \cdot z)} \nonumber \\
         &= \EE_{z \in \TT_m} \abs{\bigp{\tau-I}(z)} \nonumber \\
         &= \frac{\abs{\TT_m \cap (\partial I + \delta)}}{m} \EE_{z \in \TT_m \cap (\partial I + \delta)} \abs{\bigp{\tau-I}(z)} + \frac{\abs{\TT_m \cap (\partial I + \delta)^c}}{m} \EE_{z \in \TT_m \cap (\partial I + \delta)^c} \abs{\bigp{\tau-I}(z)} \nonumber \\
         &\le \frac{\abs{\TT_m \cap (\partial I + \delta)}}{m} + \delta. \label{eq:torusbound}
    \end{align}
    To bound the first term, let us bound the number of points of $\TT_m$ contained in any closed interval $J \subset \TT$. 
    \begin{claim}\label{cl:pointpacking}
        Let $J \subset \TT$ be a closed interval. Then the proportion of points of $\TT_m$ inside $J$ is between $\abs{J}-\frac 1 m$ and $\abs{J} + \frac 1 m$.
    \end{claim}
    \begin{proof}
        The number of points of $\TT_m$ inside $J$ is clearly maximized if one point lands exactly at an endpoint of $J$, in which case there are $1+ \floor {|J|/(1/m)}$ points inside $J$. This number is minimized if one point is close to an endpoint but outside $J$, in which case it is $\floor{\abs{J}m}$.
    \end{proof}
    Note that $\partial I + \delta$ is the union of two intervals each of length $2\delta$. Applying the upper bound in the claim to its closure then gives that \eqref{eq:torusbound} is at most
    $$
    2\bigp{2\delta + \frac 1 m} + \delta \le 5\bigp{\delta + \frac 1 m}.
    $$
\end{proof}

The following proposition assumes $Z$ is torsion sparse in the sense that $Z[a]$ is small for all small $a$, forcing many elements to have high order which the smoothing is then able to handle.

\begin{proposition}[Good pairs]\label{prop:bad}
    Let $\delta = \delta(N) = O((\log N)^{-2})$ be positive. Let $Z$ be an abelian group with $\abs{Z} = N$ such that for every positive integer $a\le \delta^{-2}$, we have $Z[a] < N^{1/4}$. Then at least a $1-O(\delta^{-5}N^{-3/4})$ fraction of ordered pairs $(x,y) \in Z^2$ are $\delta$-good.
\end{proposition}
\begin{proof}
Set $m = 1/\delta$. We will consider sampling $x,y$ uniformly and independently with replacement and show that with probability at least $1-O(\delta^{-5}N^{-3/4})$ the pair $x,y$ is $\delta$-good. By the definition of $\delta$-good and the triangle inequality it suffices to bound each of the following quantities by $20 \delta$ with probability at least $1-O(\delta^{-5}N^{-3/4})$:
\begin{align}
    \abs{\EE_z I(x \cdot z)I(y \cdot z) - \EE_z \tau(x \cdot z)\tau(y \cdot z)},\label{eq:first}\\
\abs{\EE_z \tau(x \cdot z)\tau(y \cdot z) - \EE_z \tau(x \cdot z) \EE_z \tau(y \cdot z)},\label{eq:second}\\
\abs{\EE_z \tau(x \cdot z) \EE_z \tau(y \cdot z) - \EE_z I(x \cdot z) \EE_z I(y \cdot z)}. \label{eq:third}
\end{align}

For \eqref{eq:first} and \eqref{eq:third}, we use the inequality
$$
\abs{ab-cd} = \abs{b(a-c)+c(b-d)} \le \abs{b}\abs{a-c} + \abs{c}\abs{b-d}.
$$
If $x$ and $y$ both have order at least $m = 1/\delta$, then applying this to \eqref{eq:first} along with \Cref{lem:L1} and the fact that $\tau,I$ are 1-bounded gives that
\begin{align*}
    &\abs{\EE_z I(x \cdot z)I(y \cdot z) - \EE_z \tau(x \cdot z)\tau(y \cdot z)} \le \EE_z \abs{I(x \cdot z)I(y \cdot z)-\tau(x \cdot z)\tau(y \cdot z)}\\
    &\le \EE_z \abs{I(y \cdot z)}\abs{I(x \cdot z) - \tau(x \cdot z)} + \EE_z \abs{\tau(x \cdot z)}\abs{I(y \cdot z) - \tau(y \cdot z)} \le 20\delta.
\end{align*}
Similarly for \eqref{eq:third},
\begin{align*}
    &\abs{\EE_z \tau(x \cdot z) \EE_z \tau(y \cdot z) - \EE_z I(x \cdot z) \EE_z I(y \cdot z)} \\
    &\le \abs{\EE_z \tau(y \cdot z)}\abs{\EE_z (\tau-I)(x \cdot z)} + \abs{\EE_z I(x \cdot z)}\abs{\EE_z (\tau- I)(y \cdot z)} \le 20\delta.
\end{align*}
To bound the number of elements of order less than $m$, we use our torsion-sparse assumption. Indeed the number of such elements is at most
$$
\sum_{a=1}^m Z[a] < mN^{1/4},
$$
as $m = \delta^{-1} \le \delta^{-2}$ for sufficiently large $N$. Thus by the union bound the probability that either of $x,y$ has order less than $m$ is at most $(2N^{1/4})/(\delta N)= O(\delta^{-1}N^{-3/4}) = O(\delta^{-5}N^{-3/4})$ as desired.

It remains to bound \eqref{eq:second}. For $f,g: \TT \to \CC$ equal to either $\tau$ or $1$ and for $x,y \in Z$, we may expand
    \begin{align*}\label{eq:expansion}
        \EE_{z \in Z} f(x \cdot z)g(y \cdot z) - \EE f \EE g &= \sum_{\substack{a,b \in \ZZ \\ (a,b) \neq (0,0)}} \widehat{f}(a)\widehat{g}(b) \EE_z e(z \cdot (ax + by)) \nonumber\\
        &= \sum_{\substack{a,b \in \ZZ \\ (a,b) \neq (0,0)}} \widehat{f}(a)\widehat{g}(b) 1_{ax+by = 0},
    \end{align*}
    where we have used the fact that $\cdot$ is symmetric, bilinear, and nondegenerate. Applying this three times, with $f=g= \tau$, then $f=\tau,g \equiv 1$, and finally $f \equiv 1, g = \tau$ gives that
\begin{align*}
        &\EE_z \tau(x \cdot z)\tau(y \cdot z) - \EE_z \tau(x \cdot z)\EE_z \tau(y \cdot z)\\
        &= \sum_{\substack{a,b \in \ZZ \\ (a,b) \neq (0,0)}} \widehat{\tau}(a)\widehat{\tau}(b) \bigp{1_{ax+by = 0} - 1_{ax=0}1_{by=0}}.
\end{align*}
To apply Markov's inequality, let us bound the first moment. Recalling that $x,y$ are independent and uniformly distributed, $ax$ is uniformly distributed on $aZ$ and $by$ is independently uniform on $bZ$ (fibers of $x \to ax$ are all cosets of size $Z[a]$, and similarly for $b$). Then
    $$
    \EE_{xy} \bigp{1_{ax+by = 0} - 1_{ax=0}1_{by=0}} = \frac{\abs{aZ \cap bZ}}{\abs{aZ}\abs{bZ}} - \frac{1}{\abs{aZ}\abs{bZ}},
    $$
    because the only pairs $ax,by \in aZ \times bZ$ that sum to zero are of the form $w,-w \in aZ \cap bZ$.
    Then by the Fourier decay estimate of \Cref{prop:approx},
    \begin{equation}\label{eq:firstmom}
        \EE_{x,y} \abs{\sum_{\substack{a,b \in \ZZ \\ (a,b) \neq (0,0)}} \widehat{\tau}(a)\widehat{\tau}(b) \bigp{1_{ax+by = 0} - 1_{ax=0}1_{by=0}}} \le 4 \sum_{\substack{a,b \in \NN}} \frac{\abs{aZ \cap bZ}-1}{\abs{aZ}\abs{bZ}} e^{-\delta^3(a^2 + b^2)},
    \end{equation}
    where the terms with $a=0$ or $b=0$ are zero as the indicators cancel. Now note that
    $$
    \frac{\abs{aZ \cap bZ}-1}{\abs{aZ}\abs{bZ}} \le \min\bigp{\frac{1}{\abs{aZ}}, \frac{1}{\abs{bZ}}} \le \frac{1}{\sqrt{\abs{aZ}\abs{bZ}}},
    $$
    so that \eqref{eq:firstmom} is at most
    $$
    4 \bigp{\sum_{\substack{a \in \NN}} \abs{aZ}^{-1/2} e^{-\delta^3a^2}}^2.
    $$
    The tail of the inner sum decays quickly regardless of $Z$, as by the Gaussian integral tail bound \Cref{lem:gaussint} with $s = 1/\sqrt{\delta}$
    $$
    \sum_{a > \delta^{-2 }}\abs{aZ}^{-1/2} e^{-\delta^3a^2} \le \int_{\delta^{-2}}^\infty e^{-\delta^3x^2} dx = \delta^{-3/2} \int_{\delta^{-1/2}}^\infty e^{-u^2} du \le 2\delta^{-3/2} e^{-\frac{3}{4\delta}} \le e^{-1/(2\delta)}
    $$
    for sufficiently small $\delta$. Note also that $Z[a]|aZ| = N$ by considering the homomorphism $z \to az$, so that for every $a \le \delta^{-2}$ we have $|aZ| \ge N^{3/4}$ by the torsion-sparse assumption. Thus \eqref{eq:firstmom} is at most
    \begin{equation*}\label{eq:firstmombound}
        4 \bigp{e^{-1/(2\delta)} + \delta^{-2}N^{-3/8}}^2 = O(\delta^{-4}N^{-3/4}),
    \end{equation*}
    as $e^{-1/(2\delta)}$ decays super-polynomially in $N$ for $\delta = O((\log N)^{-2})$.
    Markov's inequality then yields that the probability that \eqref{eq:second} is at least $20\delta$ is $O(\delta^{-5}N^{-3/4})$.

\end{proof}
\begin{remark}
    One can see here the use of the torsion-sparse assumption. If $Z=\ZZ_2^n \oplus \ZZ_3$, then the $a=b=2$ term in the right-hand side of \eqref{eq:firstmom} is
    $$
    \frac{\abs{2Z}-1}{\abs{2Z}^2} e^{-8\delta^3} = \frac 2 9 e^{-8\delta^3} \ge \frac{1}{5}
    $$
    for sufficiently large $N$, so \eqref{eq:firstmom} would not decay at all. This example is a bit nuanced, as one could use Alon's trick to project onto $\ZZ_2^n$ and eliminate this obstruction. However, there are more complex examples that involve taking direct sums of elementary $p$-groups of the first $\log \log N$ primes $p$ that exhibit this same obstruction and for which projection cannot help.
\end{remark}

We now recall and prove the torsion-sparse coloring theorem, which in particular shows that $3$-colorability with $\Omega(\log N)$ generators is attainable for any torsion-sparse abelian group. In fact, in this regime, we can consider unbounded chromatic number, and obtain good dependence of the constant on the number of colors (see \Cref{prop:hoffman}).

\td*
\begin{proof}
    Let $S$ denote the random generating set. Set $\alpha \coloneqq \abs{I} = 1- 2/\chi$. We will show that with high probability there is some $x$ such that $S \subset B_x$, which also implies that $-S \subset B_x$ by symmetry of $I$. If this is the case then a proper coloring is given by partitioning $\TT$ into $\chi$ pairwise disjoint intervals of length $1/\chi$ and labeling $z \in Z$ with the interval of $\TT$ that $x \cdot z$ lies in. Any two vertices $z,z'$ sharing this label have the property that $x \cdot (z - z') \in [-1/\chi,1/\chi]$, so $z - z' \notin S \cup (-S)$ and there is no edge between $z$ and $z'$ in the Cayley graph.

    For any $x \in Z$ of order $m$, by Claim \ref{cl:pointpacking}
    \begin{equation}\label{eq:singlegood}
        \abs{\frac{\abs{B_x}}{N} - \alpha} \le \frac 1 m.
    \end{equation}
    Now set $m \coloneqq 1/\delta \coloneqq k \log N$ and let $\cH \subset Z$ be the elements with order at least $m$. Let $X$ be the random variable counting $x \in \cH$ such that $S \subset B_x$, so it suffices to show that with high probability $X \neq 0$. In the proof of \Cref{prop:bad}, where we also set $m = 1/\delta$, it was shown that $\cH$ constitutes a $1-O(\delta^{-1}N^{-3/4})$ fraction of $Z$ under a torsion-sparse assumption of exactly the strength we now have, as $\delta^{-2} = (k\log N)^2$. Then
    $$
    \delta^{-1}N^{-3/4}= mN^{-3/4} = o(1),
    $$
    where we have used that $\chi \le N^{1/11}$ so that $m = \widetilde{O}(N^{1/11}) = o(N^{3/4})$. Thus $\cH$ is almost all of $Z$. Combining this fact with \eqref{eq:singlegood} yields that
    \begin{align}\label{eq:firstmomlower}
        \EE X &= \sum_{x \in \cH} \pr{S \subset B_x} = \sum_{x \in \cH} \bigp{\frac{\abs{B_x}}{N}}^k \ge \bigp{1-o(1)}N\bigp{\alpha - \frac 1 m}^k \nonumber \\
        &= (1-o(1))N\alpha^k = (1-o(1))N^{8/9} \to \infty
    \end{align}
    with our chosen value $k = -(9\log \alpha)^{-1} \log N$, where we have also used that $m = \omega(k)$.

    For the second moment now let $\cG$ denote the set of pairs $(x,y) \in \cH^2$ that are $\delta$-good. Then the contribution of such good pairs to $\EE X^2$ is
    \begin{align*}
        \sum_{(x,y) \in \cG} \pr{S \subset B_x \cap B_y} &= \sum_{(x,y) \in \cG} \bigp{\frac{\abs{B_{x,y}}}{N}}^k \\
        &= \sum_{(x,y) \in \cG} \bigp{\frac{\abs{B_{x}}\abs{B_y}}{N^2} + O(\delta)}^k \\
        &\le (1+O(\delta))^k \bigp{\EE X}^2\\
        &= (1+o(1))\bigp{\EE X}^2,
    \end{align*}
    where we have used that $\delta^{-1} = \omega(k)$ and the fact that $x,y \in \cH$ so that \eqref{eq:singlegood} gives $\frac{\abs{B_{x}}\abs{B_y}}{N^2} = \Omega(1)$. Turning to the $\delta$-bad pairs, by \Cref{prop:bad} there are at most $O(\delta^{-5}N^{5/4})$ of them. Thus the second moment can be split into the contribution of good and bad pairs, yielding
    $$
    \EE X^2 \le (1+o(1)) \bigp{\EE X}^2 + O\bigp{\delta^{-5}N^{5/4}} = (1+o(1))\bigp{\EE X}^2,
    $$
    provided that $\frac{O\bigp{\delta^{-5}N^{5/4}}}{\bigp{\EE X}^2} = o(1)$. Indeed, by \eqref{eq:firstmomlower}
    $$
    \frac{\delta^{-5}N^{5/4}}{\bigp{\EE X}^2} \le 2\bigp{\delta^{5}N^{3/4}\alpha^{2k}}^{-1} = 2\bigp{\delta^{5}N^{3/4}N^{-2/9}}^{-1} = o(1)
    $$
    by our choice of $k$ and assumption that $\chi \le N^{1/11}$ so that $\delta^{5}N^{3/4} = (k\log N)^{-5}N^{3/4} = \widetilde{\Omega}(N^{-5/11}N^{3/4}) = \Omega(N^{1/4})$. The second moment method then says that with high probability $X$ is close to its expectation, and hence positive by \eqref{eq:firstmomlower}. Explicitly, for any fixed $\eps > 0$, Chebyshev gives that
    $$
    \pr{\abs{X-\EE X} \ge \eps \EE X} \le \frac{\EE X^2 - \bigp{\EE X}^2}{\eps^2 \bigp{\EE X}^2} = o(1).
    $$
\end{proof}

To show that the bound of \Cref{thm:main} is asymptotically tight, and that the $\chi$ dependence in \Cref{thm:lowrank} is only at most a linear factor off, we upper bound the independence number $\alpha(Z,k)$. One should thus think of $\alpha = 1/\chi$ below. The proof is reproduced from \cite{chromcay}.

\begin{proposition}[Hoffman bound]\label{prop:hoffman}
    There exists an absolute constant $C > 0$ such that the following holds. Let $G$ be a finite group of order $N$ and $\alpha = \alpha(N)$ be such that $N^{-1/4} \le \alpha \le 1$. Then if $k = C\alpha^{-2}\log N$ generators are chosen uniformly, with high probability $\alpha(G,k) \le \alpha N$.
\end{proposition}
\begin{proof}
    It is shown in \cite{alonroich} that there exists some constant $C'$ such that with high probability all nontrivial eigenvalues of $(G,k)$ are at most $C'\sqrt{k\log N}$ in magnitude. Note also that $(G,k)$ is an undirected regular graph with degree $d \le 2k$, and with high probability we may assume $d \ge k-O(k^2/N) \ge k/2$ for sufficiently large $N$ by our lower bound on $\alpha$. Hoffman's ratio bound \cite{hoffman} on the independence number of a $d$-regular graph with smallest (most negative) eigenvalue $\lambda_N$ then says that
    $$
    \alpha(G,k) \le N \frac{-\lambda_N}{d-\lambda_N} \le N \frac{C'\sqrt{k\log N}}{k/2} \le \frac{2C'}{\sqrt{C}}\alpha N \le \alpha N
    $$
    for $C \ge 4C'^2$.
\end{proof}

\section{Torsion-dense groups}

This section proves \Cref{thm:highrank}. In the previous regime, since $k = \Omega(\log N)$ for any choice of $\chi$ in \Cref{thm:lowrank}, one always needs to assume $Z[a]$ is small for all $a \le (\log N)^4$, up to constants. Thus in this regime we will be assuming there exists an integer $a$ such that $a \le (\log N)^4$ and $Z[a] \ge N^{1/4}$, although the specific values of $(\log N)^4$ and $N^{1/4}$ are not so important for the result to hold. Such an assumption naturally points us to a large collection of homomorphisms into $\ZZ_a$ amongst which to search for one such that $f(S) \subseteq I$, where now $I$ is roughly the middle third of the discrete torus $\ZZ_M$.

In order to obtain good concentration, we search only among surjective homomorphisms $f: Z \to \ZZ_{M}$ for a carefully chosen modulus $M \mid a$, working in the discrete torus instead of $\TT$ as we will not need any smooth approximation. $M$ must be chosen to balance two things. Taking large $M$ is advantageous as it allows for more homomorphisms in our search space (although this relation is only monotone in the sense that $Z[M]$ is). On the other hand taking small $M$ is sometimes necessary to bound the number of pairs of homomorphisms that are strongly dependent, and there are examples where $M = a$ does not suffice. As will be seen shortly, the key statistic governing the dependence between $f,g: Z \to \ZZ_M$ is $\abs{g(K)}$ where $K$ is the kernel of $f$. The smaller $g(K)$ is, the more aligned $f,g$ are and the less randomness is left in $g(z)$ after conditioning on $f(z)$.

\begin{lemma}[Homomorphism count]\label{lem:highrankbad}
    Let $M$ be a positive integer and let $\cF$ be the collection of surjective homomorphisms $Z \to \ZZ_M$. Then
    $$
    \abs{\cF} \ge Z[M] - \sum_{\substack{t \mid M \\ t < M}} Z[t].
    $$
    Furthermore, if $f \in \cF$ and $K \le Z$ denotes the kernel of $f$, then the number of $g \in \cF$ such that $g(K) \subseteq d\ZZ_M$ for $d \mid M$ is at most $dZ[M/d]$.
\end{lemma}
\begin{proof} 

    Fix any $f \in \cF$ and let $K$ be its kernel. We will show that in fact the number of homomorphisms $g: Z \to \ZZ_M$ such that $g(K) \subseteq d\ZZ_M$ is equal to $dZ[M/d]$, which clearly gives the desired upper bound on surjective $g$. Fix any homomorphism $g$. Using the surjectivity of $f$, fix $u$ such that $f(u) = 1$. Then any $z \in Z$ can be written as
    $$
    z = f(z)u + k, \quad k \in K,
    $$
    where $f(z)$ is thought of as an integer. Indeed, taking $k \coloneqq z - f(z)u$,
    $$
    f(k) = f(z) - f(z)f(u) = 0.
    $$
    Then
    $$
    g(z) \equiv f(z) g(u) \mod d,
    $$
    as $g(k) \in d\ZZ_M$ and all elements of $d\ZZ_M$ are multiples of $d$ by definition. Thus $g$ is completely determined mod $d$ after choosing $g(u) \mod d$, for which we have $d$ choices. After fixing this choice, the remaining freedom in $g$ is just to add any homomorphism whose values are all multiples of $d$ (the values of $g(k)$), i.e., a homomorphism to $d\ZZ_M \cong \ZZ_{M/d}$. We claim the number of such homomorphisms is exactly $Z[M/d]$. Combining this claim with the $d$ choices for $g(u) \mod d$ would then give the desired count. This count can also be seen by projecting $\ZZ_M \to \ZZ_M/d\ZZ_M \cong \ZZ_d$ and applying the first isomorphism theorem.
    
    To show the claim, we argue that in general the number of homomorphisms $Z \to \ZZ_a$ is $Z[a]$ for a positive integer $a$. Write
    $$
    Z = \bigoplus_n \ZZ_n.
    $$
    Then $Z[a] = \prod_n \ZZ_n[a]$ as the order of an element divides $a$ if and only if the order of each of its coordinates divides $a$. For a given $n$, it is easy to check that $\ZZ_n[a] = (n,a)$. Turning to the number of homomorphisms $Z \to \ZZ_a$, it can be computed by multiplying the number of homomorphisms $\ZZ_n \to \ZZ_a$ across $n$. But the number of homomorphisms $\ZZ_n \to \ZZ_a$ is also equal to $(n,a)$, so the factors of both products come in equal pairs.

    We turn now to the total number of surjective homomorphisms. The number of all homomorphisms $Z \to \ZZ_M$ is just $Z[M]$ by the above. For a homomorphism not to be surjective, it must map $Z$ to a proper subgroup of $\ZZ_M$. These proper subgroups are given by $\ZZ_t$ for divisors $t$ of $M$ with $t < M$, so the result follows by the union bound and the above fact that the number of homomorphisms $Z \to \ZZ_t$ is $Z[t]$.
\end{proof}

The following ``fact about numbers'' will be used to choose the modulus $M$ with desirable properties. It does not use any property of $Z[\cdot]$ as a function except that it is positive and $Z[1] = 1$.

\begin{lemma}\label{lem:powerpicker}
    Suppose there exist positive constants $c,C$ such that $Z[a] \ge N^{c}$ for some positive integer $a \le (\log N)^{C}$. Then there exists a divisor $M$ of $a$ with $M \neq 1$ such that for every proper divisor $1 \le t < M$ of $M$ it holds that
    \begin{equation}\label{eq:ugap}
        \log Z[M] - \log Z[t] \ge \frac{\log Z[a]}{4t} + \sqrt{\log N}
    \end{equation}
    for $N \ge N_0(c,C)$.
\end{lemma}
\begin{proof}
    Suppose this is not the case, and so for every $1 \neq m \mid a$ we may find a divisor $1 \le t(m) < m$ of $m$ such that
    $$
    \log Z[m] - \log Z[t(m)] < \frac{\log Z[a]}{4t(m)} + \sqrt{\log N}.
    $$
    Then setting $m_0 = a$, one can construct a strictly decreasing chain $m_{i+1} = t(m_i)$ terminating at $m_n = 1$ (note that $\log Z[1] = 0$). Then
    $$
    \log Z[a] = \sum_{i=1}^{n} \log Z[m_{i-1}] - \log Z[m_{i}] < \frac{\log Z[a]}{4} \sum_{i=1}^{n} \frac{1}{m_{i}} + \sum_{i=1}^{n} \sqrt{\log N}.
    $$
    Consider the first sum. Since $m_{i+1}$ is a proper divisor of $m_i$ it holds that $m_i \ge 2m_{i+1}$ and, iterating from $m_n = 1$, we obtain $m_{n-j} \ge 2^j$. So
    $$
    \sum_{i=1}^{n} \frac{1}{m_{i}} \le \sum_{j=0}^\infty 2^{-j} = 2.
    $$
    For the second sum, note that by the same reasoning $n$ is at most the number of times $a$ needs to be divided by $2$ to reach $1$, so $n = O_C(\log \log N)$ by assumption on $a$. Thus
    $$
    \log Z[a] < \frac{\log Z[a]}{2} + O_C\bigp{\sqrt{\log N} \log \log N}.
    $$
    But this contradicts the assumption that $\log Z[a] \ge c\log N$ for sufficiently large $N$, depending only on $c,C$.
\end{proof}

In this regime our approach does not obtain good dependence on the number of colors $\chi$, so the result is only stated for the smallest value, $\chi = 3$. One route to obtaining better $\chi$ dependence may be to consider homomorphisms $Z \to \ZZ_M^d$ for $d > 1$, but to keep the proof as simple as possible (and since bounded $\chi$ is our main concern) we do not pursue this here.

\tdense*
\begin{proof}
    Let $S$ denote the random generating set and $k= (\log N)/64$ the number of generators sampled. Invoke \Cref{lem:powerpicker} with the $a$ given by assumption to obtain the modulus $M \mid a$ satisfying \eqref{eq:ugap}. Set $J \coloneqq \{-(\ceil{M/3}-1), \dots, \ceil{M/3}-1\}$ and let $I \coloneqq \ZZ_M \setminus J$ be (roughly) the middle third with $\alpha = |I|/M$. We will show that with high probability there is some surjective homomorphism $f$ such that $f(S) \subseteq I$, so also $f(-S) \subset I$ by symmetry of $I$. If $M$ is $2$ or $4$ then we will in fact be able to 2-color by bipartitioning $\ZZ_M$. If $M$ is not $2$ or $4$ then tripartition $\ZZ_M$ into three intervals of nearly equal size (either $\floor{M/3}$ or $\ceil{M/3}$). Then we may color $z \in Z$ with the part that $f(z)$ lies in. Any two vertices $z,z'$ sharing a color have the property that $f(z - z') \in J$ by its definition, so $z - z' \notin S \cup (-S)$ and there is no edge between $z$ and $z'$ in the Cayley graph.
        
    Let us proceed with the second moment method. We begin by lower-bounding the first moment. Let $\cF$ denote the collection of surjective homomorphisms $f: Z \to \ZZ_{M}$ and $X$ be the random variable counting $f \in \cF$ such that $f(S) \subset I$. Note that, by \eqref{eq:ugap},
    $$
    \sum_{\substack{t \mid M \\ t < M}} \frac{Z[t]}{Z[M]} \le Me^{-\sqrt{\log N}} \le ae^{-\sqrt{\log N}} = o(1).
    $$
    Thus, by \Cref{lem:highrankbad}, it holds that $\abs{\cF} \ge (1-o(1))Z[M]$. Note that applying \eqref{eq:ugap} with $t=1$ gives that $Z[M] \ge Z[a]^{1/4} \ge N^{1/16}$. Combining this with surjectivity of all $f \in \cF$ gives that
    \begin{align*}
        \EE X = \sum_{f \in \cF} \pr{f(S) \subseteq I} = \abs{\cF} \alpha^k &\ge (1-o(1))N^{1/16}\alpha^k \\
        &\ge (1-o(1))N^{1/16}4^{-k} \\
        &\ge (1-o(1))N^{1/16}N^{-1/32} \to \infty
    \end{align*}
    with our chosen value of $k = \frac{1}{64} \log N$, where we have used that $\alpha \ge 1/4$ (attained at $M=4$).

    For the second moment,
    $$
    \EE X^2 = \sum_{f,g \in \cF} \pr{f(S) \subseteq I,g(S) \subseteq I} = \alpha^k \sum_{f,g \in \cF}\pr{g(z) \in I | f(z) \in I}^k
    $$    
    where the probability is with respect to a uniformly random $z \in Z$. We will use \Cref{lem:highrankbad} to analyze $\pr{g(z) \in I | f(z) \in I}$ and bound the number of bad pairs where it is larger than $\alpha$. Fixing some $f \in \cF$, if $K = K(f)$ denotes the kernel of $f$, then $Z$ may be partitioned into cosets of $K$ where $f$ takes a distinct constant value over each coset. Fixing any $x \in \ZZ_M$ (and in particular in $I$), the event $f(z) = x$ thus conditions $z$ to be uniformly random in a coset $u + K$ for some $u \in Z$. Note that $g(K) \le \ZZ_M$ is a subgroup and thus isomorphic to a cyclic group. We may thus write $g(K) = d\ZZ_M$ where $d = d(f,g) \mid M$ can be thought of as encoding how much $f,g$ align, with larger $d$ indicating more alignment. Then if $z$ is uniform on a coset $u + K$, $g(z)$ is uniform on the coset $g(u) + d\ZZ_M \subseteq \ZZ_M$, which has size $M/d$. We consider separately the contributions to the second moment coming from independent pairs $f,g$ where $d(f,g) = 1$ and from dependent pairs where $d(f,g) > 1$, with the former accounting for nearly all of the second moment and the latter being negligible.
    
    Let us start with independent pairs. If $d=1$, then the coset $g(u) + \ZZ_M$ is just all of $\ZZ_M$, so by definition of $\alpha$
    $$
    \pr{g(z) \in I | f(z) \in I} = \alpha.
    $$
    The conditioning changes nothing and $f,g$ behave truly independently. Thus
    $$
    \alpha^k \sum_{f \in \cF} \sum_{\substack{g \in \cF \\ g(K) = \ZZ_M}} \pr{g(z) \in I | f(z) \in I}^k = \alpha^{2k}\sum_{f \in \cF} \abs{\set{g \in \cF: g(K) = \ZZ_M}} \le \alpha^{2k} \abs{\cF}^2.
    $$
    
    We turn now to dependent pairs. If $d > 1$, since $I$ is an interval in $\ZZ_M$ and $d\ZZ_M \cong \ZZ_{M/d}$
    $$
    \abs{\pr{g(z) \in I | f(z) \in I} - \alpha} \le \frac{d}{M},
    $$
    by conditioning on $f(z) = x$ for any $x \in I$ and applying a discrete analog of the argument in Claim \ref{cl:pointpacking}.
    Thus by \Cref{lem:highrankbad} and recalling that $\abs{\cF} \ge (1-o(1))Z[M]$
    \begin{align*}
         &\alpha^k \sum_{f \in \cF} \sum_{\substack{1 \neq d \mid M}} \sum_{\substack{g \in \cF \\ g(K) = d\ZZ_{M}}} \pr{g(z) \in I | f(z) \in I}^k \\
         &\le \alpha^k \sum_{\substack{1 \neq d \mid M}} \sum_{f \in \cF} \abs{\set{g \in \cF: g(K) = d\ZZ_{M}}}\bigp{\alpha + \frac{d}{M}}^k\\
         &\le \alpha^{2k} \abs{\cF}^2 \sum_{\substack{1 \neq d \mid M}} (1+o(1))d\frac{Z[M/d]}{Z[M]}\bigp{1 + \frac{d}{\alpha M}}^k\\
         &\lesssim \alpha^{2k} \abs{\cF}^2 \sum_{\substack{1 \neq d \mid M}} d \exp \bigp{\log Z[M/d] - \log Z[M]  + 4k\frac{d}{M}},
    \end{align*}
    where the last inequality uses $\alpha \ge 1/4$ and $1+x \le e^x$ for $x > -1$. Reparametrize $M/d \to t$, so that for any proper divisor $t$ of $M$ the exponent of the summand is
    $$
    \log Z[t] - \log Z[M] + \frac{4k}{t}.
    $$
    We can now upper bound this exponent uniformly in $t$, and hence $d$. Indeed, by \eqref{eq:ugap} this exponent is at most
    $$
    \frac{16k -\log Z[a]}{4t} - \sqrt{\log N} \le - \sqrt{\log N}
    $$
    since $\log Z[a] \ge \frac 1 4 \log N \ge 16k$. Then, because $\sum_{1 \neq d \mid M} d \le M^2 \le a^2 \le (\log N)^8$, the whole second moment contribution from dependent pairs is at most
    $$
    O\bigp{\alpha^{2k} \abs{\cF}^2 (\log N)^8 e^{- \sqrt{\log N}}} = o\bigp{\alpha^{2k}\abs{\cF}^2}.
    $$
    
    Putting the two contributions together, we have
    $$
    \EE X^2 \le \alpha^{2k} \abs{\cF}^2 + o\bigp{\alpha^{2k} \abs{\cF}^2} = (1+o(1))\alpha^{2k} \abs{\cF}^2 = (1+o(1))\bigp{\EE X}^2.
    $$
    The second moment method then says that with high probability $X$ is close to its expectation, and, in particular, positive.
\end{proof}

\begin{remark}
    Roughly speaking, we needed $\frac{Z[M/d]}{Z[M]}$ to decay rather rapidly for the above proof to work, and certainly to at least decay for every $1 \neq d \mid M$. This is where the use of the torsion-dense assumption can be seen. If $Z = \ZZ_N$ is cyclic, then
    $$
    \frac{Z[M/d]}{Z[M]} = \frac{(M/d, N)}{(M,N)} \ge \frac 1 d = \Omega(1)
    $$
    for any bounded $d$.
\end{remark}

\subsection*{Acknowledgements}

I would like to thank Daniel Altman for suggesting approximate vanishing sets as an analog of the vanishing sets considered in our joint work, which served as the starting point for the application presented here. I would also like to thank James Leng for references on Fourier smoothing and Jacob Fox and Sarah Peluse for helpful discussions.

\appendix

\section{Smooth approximation}\label{sec:smoothapprox}

This section proves \Cref{prop:approx} by convolution with a Gaussian kernel. For any $t > 0$, define the heat kernel $H_t: \TT \to \RR$ via its Fourier coefficients
\begin{equation}\label{eq:heat}
    H_t(x) \coloneqq \sum_{n \in \ZZ} e^{-4\pi^2 n^2 t}e(nx).
\end{equation}
Equivalently, $H_t$ is the periodization of the heat kernel on $\RR$ defined by
$$
\cH_t(x) \coloneqq \frac{1}{\sqrt{4 \pi t}}e^{-\frac{x^2}{4t}}.
$$
Indeed, by Poisson summation \cite{stein}
\begin{equation}\label{eq:poisum}
    H_t(x) = \sum_{n \in \ZZ} \cH_t(x+n),
\end{equation}
recalling the identification $\TT = [0,1)$. Finally,
\begin{equation}\label{eq:integral}
    \int_\TT H_t(x) dx = 1.
\end{equation}

\begin{lemma}[Gaussian integral bound]\label{lem:gaussint}
    For any $x,s > 0$,
    $$
    \int_{x}^\infty e^{-u^2} du \le \sqrt{\pi} e^{s^2/4 - sx}.
    $$
\end{lemma}
\begin{proof}
    Recognizing the density of a Gaussian random variable $X$ with mean zero and variance $1/2$,
    $$
    \int_{x}^\infty \frac{1}{\sqrt{\pi}} e^{-u^2} du = \pr{X \ge x} = \pr{e^{sX} \ge e^{sx}} \le \frac{\EE e^{sX}}{e^{sx}} \le e^{s^2/4 - sx},
    $$
    using Markov's inequality and the moment generating function of $X$.
\end{proof}

\begin{proof}[Proof of \Cref{prop:approx}]
    Define the approximation $\tau \coloneqq I \ast H_{t}$ for $t = \delta^3$. Then by \eqref{eq:heat} for any $a \in \ZZ$
    $$
    \widehat{H_t}(a) = e^{-4 \pi^2 a^2 \delta^3} \le e^{-a^2 \delta^3}.
    $$
    The Fourier coefficient bound in \Cref{prop:approx} then follows from the fact that $\lvert \widehat{I}(a) \rvert \le 1$ for all $a \in \ZZ$ and that $\widehat{\tau}(a) = \widehat{I}(a) \widehat{H_t}(a)$. 
    
    Turning to the approximation statement, consider first the case that $x \in I$ but at least $d$ away from the boundary. Then by \eqref{eq:poisum} and \eqref{eq:integral}
    \begin{align*}
         1-\tau(x) &= \int_\TT H_t(x-y)(1-I(y)) dy = \int_{\TT \setminus I} H_t(x-y)dy \\
         &\le \int_{\norm{y}_{\TT} \ge d} H_t(y) dy = \sum_{n \in \ZZ} \int_d^{1-d} \cH_t(n+y) dy \le 2 \int_d^\infty \cH_t(y) dy,
    \end{align*}
    where the last inequality is by symmetry and nonnegativity of $\cH_t$. Recognizing $\cH_t(y)$ as the density of a Gaussian with mean zero and variance $2t$ and applying a Chernoff bound with $s = d/(2t)$ (or using \Cref{lem:gaussint}) gives that this is at most $2\exp (-d^2/(4t))$. Finally plugging in our chosen value $t = \delta^3$ and taking $d = \delta$ gives a bound that is less than $\delta$ provided $\delta$ is sufficiently small. 
    
    If $x \notin I$ and is at distance at least $d$ from the boundary, then by \eqref{eq:integral}
    \begin{align*}
         \tau(x) &= \int_\TT H_t(x-y)I(y) dy = \int_{I} H_t(x-y)dy \le \int_{\norm{y}_{\TT} \ge d} H_t(y) dy.
    \end{align*}
    The rest proceeds identically to the previous case.
\end{proof}

\bibliographystyle{amsplain}
\bibliography{Biblio.bib}

\end{document}